\documentclass[12pt]{article}

\usepackage[colorlinks]{hyperref}            
\usepackage{color}
\usepackage{graphicx,subfigure,amsmath,amssymb,amsfonts,bm,epsfig,epsf,url,dsfont}
\usepackage{times}
\usepackage{bbm}      
\usepackage{booktabs}
\usepackage{cases}
\usepackage{fullpage}
\usepackage[small,bf]{caption}
\usepackage{natbib}
\usepackage[top=1in,bottom=1in,left=1in,right=1in]{geometry}
\usepackage{enumerate}

\def\eps{\varepsilon}
\newtheorem{theorem}{Theorem}
\newtheorem{lemma}{Lemma}
\newtheorem{proposition}{Proposition}

\newcommand{\BlackBox}{\rule{1.5ex}{1.5ex}}  
\newenvironment{proof}{\par\noindent{\bf Proof\ }}{\hfill\BlackBox\\[2mm]}

\renewcommand{\phi}{\varphi}

\renewcommand{\P}{\mathbb{P}}
\newcommand{\E}{\mathbb{E}}

\newcommand{\R}{\mathbb{R}}

\newcommand{\cP}{\mathcal{P}}

\def\ds1{\mathds{1}}
\renewcommand{\epsilon}{\varepsilon}

\renewcommand{\tilde}{\widetilde}

\newlength{\minipagewidth}
	\newcommand{\beq}{\begin{equation}}
	\newcommand{\eeq}{\end{equation}}
	
	\newcommand{\beqa}{\begin{eqnarray}}
	\newcommand{\eeqa}{\end{eqnarray}}
	
	\newcommand{\beqan}{\begin{eqnarray*}}
		\newcommand{\eeqan}{\end{eqnarray*}}
	
	\def\ba#1\ea{\begin{align*}#1\end{align*}} 
	\def\banum#1\eanum{\begin{align}#1\end{align}} 

\begin{document}

\title{Sampling from a log-concave distribution with Projected Langevin Monte Carlo}

\author{S\'ebastien Bubeck
	\thanks{Microsoft Research; \texttt{sebubeck@microsoft.com}.}
	\and
	Ronen Eldan
	\thanks{Weizmann Institute; \texttt{roneneldan@gmail.com}.}
      \and
	Joseph Lehec
	\thanks{Universit{\'e} Paris-Dauphine; \texttt{lehec@ceremade.dauphine.fr}.}}
\date{\today}

\maketitle

\begin{abstract}
We extend the Langevin Monte Carlo (LMC) algorithm to compactly supported measures via a projection step, akin to projected Stochastic Gradient Descent (SGD). We show that (projected) LMC allows to sample in polynomial time from a log-concave distribution with smooth potential. This gives a new Markov chain to sample from a log-concave distribution. 
Our main result shows in particular that when the target distribution is uniform, LMC mixes in $\tilde{O}(n^7)$ steps (where $n$ is the dimension). We also provide preliminary experimental evidence that LMC performs at least as well as hit-and-run, for which a better mixing time of $\tilde{O}(n^4)$ was proved by Lov{\'a}sz and Vempala.
\end{abstract}

\section{Introduction}
Let $K \subset \R^n$ be a convex body such that $0 \in K$, $K$ contains a Euclidean ball of radius $r$, and $K$ is contained in a Euclidean ball of radius $R$. Denote $\cP_K$ for the Euclidean projection on $K$. Let $f : K \rightarrow \R$ be a $L$-Lipschitz and $\beta$-smooth convex function, that is $f$ is differentiable and statisfies $\forall x, y \in K, |\nabla f(x) - \nabla f(y)| \leq \beta |x-y|$, and $|\nabla f(x)| \leq L$. We are interested in the problem of sampling from the probability measure $\mu$ on $\R^n$ whose density with respect to the Lebesgue measure is given by:
$$\frac{d\mu}{dx} = \frac{1}{Z} \exp(-f(x)) \ds1\{x \in K\}, \;\; \text{where} \;\; Z= \int_{y \in K} \exp(-f(y)) dy .$$

In this paper we study the following Markov chain, which depends on a parameter $\eta >0$, and where $\xi_1, \xi_2, \ldots$ is an i.i.d. sequence of standard Gaussian random variables in $\R^n$:
\begin{equation} \label{eq:LMC}
\overline{X}_{k+1} = \cP_K \left( \overline{X}_k - \frac{\eta}{2} \nabla f(\overline{X}_k) + \sqrt{\eta} \xi_k \right) ,
\end{equation}
with $\overline{X}_0=0$. 

Recall that the total variation distance between two measures $\mu, \nu$ is defined as $\mathrm{TV}(\mu, \nu) = \sup_A |\mu(A) - \nu(A)|$ where the supremum is over all measurable sets $A$. With a slight abuse of notation we sometimes write $\mathrm{TV}(X, \nu)$ where $X$ is a random variable distributed according to $\mu$. The notation $v_n = \tilde{O}(u_n)$ (respectively $\tilde{\Omega}$) means that there exists $c \in \R, C>0$ such that $v_n \leq C u_n \log^c(u_n)$ (respectively $\geq$). We also say $v_n=\tilde{\Theta}(u_n)$ if one has both $v_n=\tilde{O}(u_n)$ and $v_n=\tilde{\Omega}(u_n)$. Our main result is the following:
\begin{theorem} \label{th:mainresult}
Assume that $r=1$ and let $\epsilon > 0$. Then one has $\mathrm{TV}(\overline{X}_N, \mu) \leq \epsilon$ provided that $\eta = \tilde{\Theta}(R^2 / N)$ and that $N$ satisfies the following: if $\mu$ is uniform then
$$N = \tilde{\Omega}\left( \frac{R^6 n^7}{\epsilon^{8}} \right) ,$$
and otherwise
$$N = \tilde{\Omega}\left( \frac{R^6 \max(n, RL , R\beta)^{12}}{\epsilon^{12}} \right) .$$
\end{theorem}

\subsection{Context and related works} \label{sec:context}
There is a long line of works in theoretical computer science proving results similar to Theorem \ref{th:mainresult}, starting with the breakthrough result of \cite{DFK91} who showed that the lattice walk mixes in $\tilde{O}(n^{23})$ steps. The current record for the mixing time is obtained by \cite{LV07}, who show a bound of $\tilde{O}(n^4)$ for the hit-and-run walk. These chains (as well as other popular chains such as the ball walk or the Dikin walk, see e.g. \cite{KN12} and references therein) all require a {\em zeroth-order oracle} for the potential $f$, that is given $x$ one can calculate the value $f(x)$. On the other hand our proposed chain \eqref{eq:LMC} works with a {\em first-order oracle}, that is given $x$ one can calculate the value of $\nabla f(x)$. The difference between zeroth-order oracle and first-order oracle has been extensively studied in the optimization literature (e.g., \cite{NY83}), but it has been largely ignored in the literature on polynomial-time sampling algorithms. We also note that hit-and-run and LMC are the only chains which are rapidly mixing from any starting point (see \cite{LV06}), though they have this property for seemingly very different reasons. When initialized in a corner of the convex body, hit-and-run might take a long time to take a step, but once it moves it escapes very far (while a chain such as the ball walk would only do a small step). On the other hand LMC keeps moving at every step, even when initialized in a corner, thanks for the projection part of \eqref{eq:LMC}.
\newline 

Our main motivation to study the chain \eqref{eq:LMC} stems from its connection with the ubiquitous {\em stochastic gradient descent} (SGD) algorithm. In general this algorithm takes the form $x_{k+1} = \cP_K \left( x_k - \eta \nabla f(x_k) + \epsilon_k \right)$ where $\epsilon_1, \epsilon_2, \hdots$ is a centered i.i.d. sequence. Standard results in approximation theory, such as \cite{RM51}, show that if the variance of the noise $\mathrm{Var}(\epsilon_1)$ is of smaller order than the step-size $\eta$ then the iterates $(x_k)$ converge to the minimum of $f$ on $K$ (for a step-size decreasing sufficiently fast as a function of the number of iterations). For the specific noise sequence that we study in \eqref{eq:LMC}, the variance is exactly equal to the step-size, which is why the chain deviates from its standard and well-understood behavior. We also note that other regimes where SGD does not converge to the minimum of $f$ have been studied in the optimization literature, such as the constant step-size case investigated in \cite{Pfl86, BM13}.
\newline

The chain \eqref{eq:LMC} is also closely related to a line of works in Bayesian statistics on Langevin Monte Carlo algorithms, starting essentially with \cite{RT96}. The focus there is on the unconstrained case, that is $K = \R^n$. In this simpler situation, a variant of Theorem \ref{th:mainresult} was proven in the recent paper \cite{Dal14}. The latter result is the starting point of our work. A straightforward way to extend the analysis of Dalalyan to the constrained case is to run the unconstrained chain with an additional potential that diverges quickly as the distance from $x$ to $K$ increases. However it seems much more natural to study directly the chain \eqref{eq:LMC}. Unfortunately the techniques used in \cite{Dal14} cannot deal with the singularities in the diffusion process which are introduced by the projection. As we explain in Section \ref{sec:contribution} our main contribution is to develop the appropriate machinery to study \eqref{eq:LMC}. 
\newline

In the machine learning literature it was recently observed that Langevin Monte Carlo algorithms are particularly well-suited for large-scale applications because of the close connection to SGD. For instance \cite{WT11} suggest to use mini-batch to compute approximate gradients instead of exact gradients in \eqref{eq:LMC}, and they call the resulting algorithm SGLD (Stochastic Gradient Langevin Dynamics). It is conceivable that the techniques developed in this paper could be used to analyze SGLD and its refinements introduced in \cite{AKW12}. We leave this as an open problem for future work. Another interesting direction for future work is to improve the polynomial dependency on the dimension and the inverse accuracy in Theorem \ref{th:mainresult} (our main goal here was to provide the simplest polynomial-time analysis).

\subsection{Contribution and paper organization} \label{sec:contribution}
As we pointed out above, \cite{Dal14} proves the equivalent of Theorem \ref{th:mainresult} in the unconstrained case. His elegant approach is based on viewing LMC as a discretization of the diffusion process $dX_t = dW_t - \frac{1}{2} \nabla f(X_t)$, where $(W_t)$ is a Brownian motion. The analysis then proceeds in two steps, by deriving first the mixing time of the diffusion process, and then showing that the discretized process is `close' to its continuous version. In \cite{Dal14} the first step is particularly clean as he assumes $\alpha$-strong convexity for the potential, which in turns directly gives a mixing time of order $1/\alpha$. The second step is also rather simple once one realizes that LMC can be viewed as the diffusion process $d \overline X_t = dW_t - \frac{1}{2} \nabla f(X_{\eta \lfloor \frac t \eta \rfloor})$. Using Pinsker's inequality and Girsanov's formula it is then a short calculation to show that the total variation distance between $\overline X_t$ and $X_t$ is small. 

The constrained case presents several challenges, arising from the {\em reflection} of the diffusion process on the boundary of $K$, and from the lack of curvature in the potential (indeed the constant potential case is particularly important for us as it corresponds to $\mu$ being the uniform distribution on $K$). Rather than a simple Brownian motion with drift, LMC with projection can be viewed as the discretization of {\em reflected Brownian motion with drift}, which is a process of the form
$dX_t = dW_t - \frac{1}{2} \nabla f(X_t) dt - \nu_t L(dt)$, where $X_t \in K, \forall t \geq 0$, $L$ is a measure supported on $\{ t \geq 0 : \ X_t \in \partial K \}$, and $\nu_t$ is an outer normal unit vector of $K$ at $X_t$. The term $\nu_t L(dt)$ is referred to as the {\em Tanaka drift}. Following \cite{Dal14} the analysis is again decomposed in two steps. We study the mixing time of the continuous process via a simple coupling argument, which crucially uses the convexity of $K$ and of the potential $f$. The main difficulty is in showing that the discretized process $(\overline X_t)$ is close to the continuous version $(X_t)$, as the Tanaka drift prevents us from a straightforward application of Girsanov's formula. Our approach around this issue is to first use a geometric argument to prove that the two processes are close in Wasserstein distance, and then to show that in fact for a reflected Brownian motion with drift one can deduce a total variation bound from a Wasserstein bound.
\newline

The paper is organized as follows. We start in Section \ref{sec:constant} by proving Theorem \ref{th:mainresult} for the case of a uniform distribution. We first remind the reader of Tanaka's construction (\cite{Tan79}) of reflected Brownian motion in Subsection \ref{sec:Sko}. We present our geometric argument to bound the Wasserstein distance between $(X_t)$ and $(\overline X_t)$ in Subsection \ref{sec:w1}, and we use our coupling argument to bound the mixing time of $(X_t)$ in Subsection \ref{sec:mixing}. Then in Subsection \ref{sec:w1totv} we use properties of reflected Brownian to show that one can obtain a total variation bound from the Wasserstein bound of Subsection \ref{sec:w1}. We conclude the proof of the first part of Theorem \ref{th:mainresult} in Subsection \ref{sec:finishingtheproof1}. In Section \ref{sec:general} we generalize these arguments to an arbitrary smooth potential. 
Finally we conclude the paper in Section \ref{sec:exp} with some preliminary experimental comparison between LMC and hit-and-run.

\section{The constant potential case} \label{sec:constant}
In this section we prove Theorem \ref{th:mainresult} for the case where $\mu$ is uniform, that is $\nabla f = 0$. First we introduce some useful notation. For a point $x \in \partial K$ we say that $\nu$ is an outer unit normal vector at $x$ if $|\nu| =1$ and
\[
\langle x-x' , \nu \rangle \geq 0 , \quad \forall x'\in K . 
\]
For $x \notin \partial K$ we say that $0$ is an outer unit normal at $x$. Let $\Vert \cdot \Vert_K$ be the gauge of $K$ defined by
\[
\Vert x \Vert_K = \inf \{ t \geq 0 ; \ x \in t K \} , \quad x \in \R^n ,
\]
and $h_K$ the support function of $K$ by
\[
h_K (y) = \sup \left\{ \langle x ,y \rangle ; \ x \in K \right\} , \quad y \in \R^n .
\]
Note that $h_K$ is also the gauge function of the polar body of $K$. Finally we denote $m = \int |x| \mu(dx)$, and $M = \E \left[ \Vert \theta \Vert_K \right]$, where $\theta$ is uniform on the sphere $\mathbb S^{n-1}$.

\subsection{The Skorokhod problem} \label{sec:Sko}
Let $T\in \R_+\cup \{+\infty\}$ and $w \colon [0,T) \to \R^n$ be a piecewise continuous
path with $w(0) \in K$. We say that $x\colon [0,T)\to \R^n$ 
and $\phi \colon [0,T) \to \R^n$ solve the Skorokhod 
problem for $w$ if one has $x(t) \in K, \forall t \in [0,T)$,
$$x(t)  = w (t) + \phi (t) , \quad \forall t \in [0,T) ,$$
and furthermore $\phi$ is of the form  
\[
\phi ( t ) = - \int_0^t \nu_s \, L(ds) , \quad \forall t \in [0,T) ,
\]
where $\nu_s$ is an outer unit normal at $x(s)$, and $L$ is a measure on $[0,T]$ supported on the set $\{ t \in [0,T) : \ x(t) \in \partial K \}$.

The path $x$ is called the \emph{reflection} 
of $w$ at the boundary of $K$, and the
measure $L$ is called the \emph{local time} 
of $x$ at the boundary of $K$.  Skorokhod showed the existence of such a a pair $(x,\phi)$ in dimension $1$ in \cite{Sko61}, and Tanaka extended this result to convex sets in higher dimensions in \cite{Tan79}. Furthermore Tanaka also showed that the solution is unique, and if $w$ is continuous then so is $x$ and $\phi$. In particular the reflected Brownian motion in $K$, denoted $(X_t)$, is defined as the reflection of the standard Brownian motion $(W_t)$ at the boundary of $K$ (existence follows by continuity of $W_t$). Observe that by It\^o's formula,
for any smooth function $g$ on $\R^n$,
\begin{equation} \label{eq:ito}
g(X_t) - g ( X_0) = \int_0^t \langle \nabla g ( X_s ) , d W_s \rangle  + \frac 12 \int_0^t \Delta g (X_s) \, ds - \int_0^t \langle \nabla g ( X_s ) , \nu_s \rangle \, L(ds). 
\end{equation}

To get a sense of what a solution typically looks like, let us work out the case where $w$ is piecewise constant (this will also be useful to realize that LMC can be viewed as the solution to a Skorokhod problem). For a sequence $g_1 \hdots g_{N} \in \R^n$, and for $\eta > 0$, we consider the path: 
\[
w(t) = \sum_{k=1}^{N} g_k \, \mathds1\{t \geq k \eta\} , \qquad t \in [0, (N+1) \eta) .
\]
Define $(x_k)_{k=0,\hdots,N}$ inductively by $x_0=0$ and
$$x_{k+1} = \cP_K ( x_k + g_k) .$$
It is easy to verify that the solution to the Skorokhod problem for $w$ is given by $x(t) = x_{ \lfloor \frac t \eta \rfloor }$
and $\phi ( t ) = - \int_0^t \nu_s \, L(ds)$, where
the measure $L$ is defined by (denoting $\delta_s$ for a dirac at $s$)
\[ 
L = \sum_{k=1}^{N} |x_k + g_k - \cP_K ( x_k + g_k)| \delta_{k \eta} ,
\]
and for $s=k \eta$,
\[ 
\nu_{s} = \frac{x_k + g_k - \cP_K ( x_k + g_k)}{|x_k + g_k - \cP_K ( x_k + g_k)|} .
\]

\subsection{Discretization of reflected Brownian motion} \label{sec:w1}
Given the discussion above, it is clear that when $f$ is a constant function, the chain \eqref{eq:LMC} can be viewed as the reflection $(\overline{X}_t)$ of a discretized Brownian motion $\overline{W}_t := W_{ \eta \lfloor \frac t \eta \rfloor }$ at the boundary of $K$ (more precisely the value of $\overline{X}_{k \eta}$ coincides with the value of $\overline{X}_k$ as defined by \eqref{eq:LMC}). It is rather clear that the discretized Brownian motion $(\overline{W}_t)$ is ``close'' to the path $(W_t)$, and we would like to carry this to the reflected paths $(\overline{X}_t)$ and $(X_t)$. The following lemma extracted from \cite{Tan79} allows to do exactly that.
\begin{lemma}\label{lem:uniq}
Let $w$ and $\overline w$ be piecewise 
continuous path and assume that $(x,\phi)$ and $(\overline x, \overline \phi)$
solve the Skorokhod problems for $w$ and $\overline w$, respectively.
Then for all time $t$ we have
\[
\begin{split}
\vert x(t) - \overline x(t) \vert^2 
& \leq \vert w(t) - \overline w(t) \vert^2 \\
& + 2 \int_0^t \langle w(t) - \overline w(t) - w(s) + \overline w(s) , \phi (ds) - \overline \phi (ds) \rangle .
\end{split} 
\]
\end{lemma}
In the next lemma we control 
the local time at the boundary of 
the reflected Brownian motion $(X_t)$.  
\begin{lemma}\label{lem:local}
We have, for all $t>0$
\[
\E \left[ \int_0^t h_K ( \nu_s ) \, L(ds) \right]  
\leq \frac{nt}2 . 
\]
\end{lemma}
\begin{proof}
By It\^o's formula
\[
d \vert X_t \vert^2  = 2 \langle X_t , d W_t \rangle 
 + n \, dt - 2 \langle X_t , \nu_t \rangle \, L(dt) .   
\]
Now observe that by definition of the reflection, if $t$ is in the support of $L$
then
\[
\langle X_t , \nu_t \rangle 
\geq \langle x, \nu_t \rangle , \quad \forall x \in K .
\]
In other words $\langle X_t , \nu_t \rangle \geq h_K ( \nu_t )$. 
Therefore
\[
2 \int_0^t h_K ( \nu_s ) \, L(ds)
\leq 2 \int_0^t \langle X_s , d W_s \rangle  + n t + \vert X_0 \vert^2 - \vert X_t \vert^2 . 
\]
The first term of the right--hand side is a martingale, 
so using that $X_0 =0$ and taking expectation 
we get the result. 
\end{proof}
\begin{lemma}\label{lem:doob}
There exists a universal constant $C$ such that
\[
\E \left[ \sup_{[0,T]} \Vert W_t - \overline W_t \Vert_K \right]
\leq C \ M \ n^{1/2} \eta^{1/2} \log(T/\eta)^{1/2} .
\]
\end{lemma}
\begin{proof} 
Note that 
\[
\E \left[ \sup_{[0,T]} \Vert W_t - \overline W_t \Vert_K \right]
= \E \left[ \max_{0\leq i\leq N-1} Y_i \right]
\]
where
\[
Y_i = \sup_{t \in [i\eta,(i+1)\eta)} \Vert W_t - W_{i\eta} \Vert_K  .
\]
Observe that the variables $(Y_i)$ are identically distributed,
let $p\geq 1$ and write
\[
\E \left[ \max_{i\leq N-1} Y_i \right]
\leq \E \left[ \left(\sum_{i=0}^{N-1} \vert Y_i \vert^p \right)^{1/p} \right] 
\leq N^{1/p} \, \Vert Y_0 \Vert_p . 
\]
We claim that 
\begin{equation} \label{eq:toprove1}
\Vert Y_0 \Vert_{p} \leq C \sqrt{p \, n \, \eta} \, M  
\end{equation}
for some constant $C$, and for all $p\geq 2$. 
Taking this for granted and choosing $p=\log (N)$
in the previous inequality yields the result 
(recall that $N = T/\eta$).  
So it is enough to prove \eqref{eq:toprove1}.  
Observe that since $(W_t)$ is a martingale, the process
\[
M_t = \Vert W_t \Vert_K 
\] 
is a sub--martingale.  
By Doob's maximal inequality
\[
\Vert Y_0 \Vert_p 
= \Vert \sup_{ [0,\eta] } M_t \Vert_p 
\leq 2 \Vert M_\eta \Vert_p ,  
\]
for every $p\geq 2$. 
Letting $\gamma_n$ be the 
standard Gaussian measure on $\R^n$
and using Khintchin's inequality
we get
\[
\begin{split}
\Vert M_\eta \Vert_p 
& = \sqrt{\eta} \left( \int_{\R^n} \Vert x \Vert_K^p \, \gamma_n ( dx ) \right)^{1/p} \\
& \leq C \sqrt{p\eta}  \int_{\R^n} \Vert x \Vert_K \, \gamma_n ( dx )
\end{split}
\]
Lastly, integrating in polar coordinate, it is easily seen that 
\[
\int_{\R^n} \Vert x \Vert_K \, \gamma_n ( dx )
\leq C  \sqrt{n} \, M . 
\]
Hence the result. 
\end{proof}
We are now in a position to bound the average
distance between $X_T$ and its discretization 
$\overline X_T$.
\begin{proposition}\label{prop:w1estimate}
There exists a universal constant $C$ 
such that for any $T \geq 0$ we have
\[
\E[ \vert X_T - \overline X_T \vert ]
\leq C \left( \eta \log(T / \eta) \right)^{1/4} n^{3/4} \, T^{ 1/2 } \, M^{1/2}  
\]
\end{proposition}
\begin{proof}
Applying Lemma~\ref{lem:uniq} 
to the processes $(W_t)$ and $(\overline W_t)$ at time $T = N \eta$ yields (note that $W_T = \overline W_T$)
\[
\vert X_T - \overline X_T \vert^2  
\leq 2 \int_0^T \langle W_t - \overline W_t , \nu_t \rangle L(dt)
- 2 \int_0^T \langle W_t - \overline W_t , \overline \nu_t \rangle \overline L(dt)
\]
We claim that the second integral is equal to $0$. 
Indeed, since the discretized process is constant 
on the intervals $[k\eta , (k+1)\eta )$ the local time $\overline L$ 
is a positive combination of Dirac point masses at 
\[
\eta , 2 \eta , \dotsc, N \eta. 
\]
On the other hand $W_{k\eta} = \overline W_{k\eta}$ for all integer $k$,
hence the claim. Therefore
\[
\vert X_T - \overline X_T \vert^2  
\leq 2 \int_0^T \langle W_t - \overline W_t , \nu_t \rangle \, L(dt)
\]
Using the inequality $\langle x , y \rangle \leq \Vert x \Vert_K \, h_K (y)$ we get
\[
\vert X_T - \overline X_T \vert^2  
\leq 2 \sup_{ [0,T] } \Vert W_t - \overline W_T \Vert_K \, \int_0^T h_K ( \nu_t ) \, L(dt). 
\]
Taking the square root, expectation and using Cauchy--Schwarz we get
\[
\E \left[ \vert X_T - \overline X_T \vert \right]^2
\leq 2 \, \E \left[ \sup_{ [0,T] } \Vert W_t - \overline W_T \Vert_K \right] \,
\E\left[  \int_0^T h_K ( \nu_t ) \, L(dt) \right] .
\]
Applying Lemma~\ref{lem:local} and Lemma~\ref{lem:doob},
we get the result. 
\end{proof}
\subsection{A mixing time estimate for the reflected Brownian motion} \label{sec:mixing}
The reflected Brownian motion is a Markov 
process. We let $(P_t)$ be the associated semi--group:
\[
P_t f (x) = \E_x [ f(X_t) ] ,
\]
for every test function $f$,
where $\E_x$ means conditional expectation 
given $X_0 = x$. It\^o's formula
shows that the generator of the semigroup $(P_t)$ 
is $(1/2) \Delta$ with Neumann boundary condition. 
Then by Stokes' formula, it is easily seen that $\mu$ 
(the uniform measure on $K$ normalized to be a probability measure) 
is the stationary measure of this process, and is even reversible. 
In this section we estimate the total variation between the law of $(X_t)$
and $\mu$. 

Given a probability 
measure $\nu$ supported on $K$, we let 
$\nu P_t$ be the law of $X_t$ when $X_0$ as law $\nu$.
The following lemma is the key result to estimate the 
mixing time of the process $(X_t)$.
\begin{lemma}\label{lem:mixing}
Let $x,x'\in K$
\[
\mathrm{TV} ( \delta_x P_t , \delta_{x'} P_t ) \leq \frac{ \vert x-x' \vert }{ \sqrt{ 2\pi t } } . 
\]
\end{lemma}
\begin{proof}
Let $(W_t)$ be a Brownian motion
starting from $0$ and let $(X_t)$
be a reflected Brownian motion 
starting from $x$:
\begin{equation} \label{eq:coupledprocs1}
\left\{
\begin{array}{l}
X_0 = x \\
dX_t = d W_t - \nu_t \, L(dt) 
\end{array}
\right. 
\end{equation}
where $(\nu_t)$ and $L$ satisfy the appropriate conditions. 
We construct a reflected Brownian motion
$(X'_t)$ starting from $x'$ as follows.
Let 
\[
\tau = \inf \{ t \geq 0 ; \ X_t = X'_t \} ,
\]
and for $t<\tau$ let $S_t$ be the orthogonal
reflection with respect to the hyperplane $(X_t - X'_t)^\perp$. 
Then up to time $\tau$, the process $(X'_t)$
is defined by
\begin{equation} \label{eq:coupledprocs2}
\left\{
\begin{array}{l}
X'_0= x' \\
d X'_t = d W'_t - \nu'_t \, L'(dt) \\
d W'_t = S_t ( d W_t ) 
\end{array}
\right. 
\end{equation}
where $L'$ is a measure supported on 
\[
\{t\leq \tau ; \ X'_t \in \partial K \}
\]
and $\nu'_t$ is an outer unit normal 
at $X'_t$ for all such $t$. 
After time $\tau$ we just set $X'_t=X_t$.  
Since $S_t$ is an orthogonal map $(W'_t)$
is a Brownian motion and thus $(X'_t)$ is 
a reflected Brownian motion starting from $x'$. 
Therefore 
\[
\mathrm{TV} ( \delta_x P_t , \delta_{x'} P_t ) 
\leq \P ( X_t \neq X'_t ) = \P ( \tau > t ). 
\]
Observe that on $[0,\tau)$ 
\[
d W_t- d W'_t = (\mathrm{I}-S_t)(d W_t) = 2 \langle V_t , d W_t \rangle V_t ,
\]
where 
\[
V_t = \frac { X_t - X'_t }{ \vert X_t - X'_t \vert } . 
\]
So
\[
\begin{split}
d (X_t - X'_t) 
& = 2 \langle V_t , d W_t \rangle V_t - \nu_t \, L(dt) + \nu'_t \, L'(dt) \\
& = 2 ( d B_t ) \, V_t - \nu_t \, L(dt) + \nu'_t \, L'(dt) ,
\end{split}
\]
where 
\[
B_t = \int_0^t \langle V_s , d W_s \rangle , \quad \text{on } [0,\tau) . 
\]
Observe that $(B_t)$ is a one--dimensional Brownian motion. 
It\^o's formula then gives
\[
\begin{split}
d g (X_t-X_t') & = 2 \langle \nabla g ( X_t - X'_t ) , V_t \rangle  \, d B_t  
- \langle \nabla g ( X_t -X'_t ) ,  \nu_t \rangle \, L(dt) \\ 
& + \langle \nabla g ( X_t - X'_t ) , \nu't \rangle \, L'(dt) 
 + 2 \nabla^2 g ( X_t - X'_t )  ( V_t , V_t ) \, dt , 
\end{split}
\]
for every $g$ which is smooth in a neighborhood of $X_t - X'_t$. 
Now if $g(x) = \vert x \vert$ then 
\[
\nabla g ( X_t - X'_t ) = V_t 
\]
so
\begin{equation}\label{eq:something}
\begin{split}
& \langle \nabla g ( X_t - X'_t ) , V_t \rangle  = 1 \\ 
& \langle \nabla g ( X_t - X'_t ) , \nu_t \rangle \geq 0 , \quad \text{on the support of } L  \\
& \langle \nabla g ( X_t - X'_t ) , \nu'_t \rangle  \leq 0 , \quad \text{on the support of } L'.
\end{split}
\end{equation}
Moreover
\[
\nabla^2 g ( X_t - X'_t) = \frac 1 { \vert X_t - Y_t \vert} \, P_{(X_t-Y_t)^{\perp}}
\] 
where $P_{x^\perp}$ denotes the orthogonal projection on $x^\perp$. 
In particular 
\[
\nabla^2 g ( X_t - Y_t) (V_t) = 0 .
\] 
We obtain
\[
\vert X_t - X'_t \vert \leq \vert x - x' \vert + 2 B_t , \quad \text{on } [0,\tau).  
\]
Therefore 
\[
\P ( \tau > t ) \leq \P ( \tau' > t ) 
\]
where $\tau'$ is the first time the Brownian motion $(B_t)$
hits the value $-\vert x-x'\vert/2$. Now by the reflection
principle
\[
\P ( \tau' > t ) = 2 \, \P \left( 0 \leq 2\, B_t < \vert x-x'\vert \right) 
\leq \frac { \vert x-x' \vert }{ \sqrt{2\pi t} } . 
\] 
Hence the result. 
\end{proof}
The above result clearly implies that
for a probability measure $\nu$ on 
$K$,
\[
\mathrm{TV} ( \delta_0 P_t , \nu P_t ) 
\leq \frac{ \int_K \vert x \vert \, \nu(dx) }{ \sqrt{2\pi t} } . 
\]
Since $\mu$ is stationary, we obtain
\begin{equation}\label{eq:mixingm}
\mathrm{TV} ( \delta_0 P_t , \mu ) 
\leq  \frac {m}{ \sqrt{2\pi t} } 
\end{equation}
for any $t>0$.
In other words, starting from $0$, the mixing time of $(X_t)$ is of order at most $m^2$.
Notice also that Lemma~\ref{lem:mixing} allows to bound the mixing time from any starting point:
for every $x\in K$, we have
\[
\mathrm{TV} ( \delta_x P_t , \mu ) \leq \frac{R}{\sqrt{2\pi t}} ,
\]
where $R$ is the diameter of $K$. Letting $\tau_{mix}$ 
be the mixing time of $(X_t)$, namely the smallest time $t$
for which 
\[
\sup_{x\in K} \{ \mathrm{TV} ( \delta_x P_t , \mu ) \} \leq \frac 1 e ,
\]
we obtain from the previous display $\tau_{mix} \leq 2 R^2$. Since
for any $x$ and $t$ we have 
$\mathrm{TV} ( \delta_x P_t , \mu ) \leq e^{- \lfloor t/\tau_{mix} \rfloor }$
(see e.g., \cite[Lemma 4.12]{LPW}) we obtain in particular 
\[
\mathrm{TV} ( \delta_0 P_t , \mu ) \leq e^{ - \lfloor t/2 R^2 \rfloor } 
\]
The advantage of this upon~\eqref{eq:mixingm} is the exponential 
decay in $t$. On the other hand, since obviously $m \leq R$, inequality~\eqref{eq:mixingm} 
can be more precise for a certain range of $t$. The next proposition sums up 
the results of this section. 
\begin{proposition}\label{prop:mixing}
For any $t>0$, we have 
\[
\mathrm{TV} ( \delta_0 P_t , \mu ) \leq C \, \min \left(  m \, t^{-1/2} , e^{-t/2R^2} \right),
\]
where $C$ is a universal constant. 
\end{proposition}
\subsection{From Wasserstein distance to total variation}\label{sec:w1totv}
In the following lemma, which is a variation on the reflection 
principle, $(W_t)$ is a Brownian motion, 
the notation $\P_x$ means probability given $W_0 =x$ and $(Q_t)$
denotes the heat semigroup: 
\[
Q_t h (x) = \E_x [ h (W_t) ] , 
\]
for every test function $h$. 
\begin{lemma}\label{lem:reflectionprinciple}
Let $x\in K$ and let $\sigma$
be the first time $(W_t)$ hits the boundary 
of $K$. Then for all $t>0$
\[
\P_x ( \sigma < t ) \leq 2 \P_x ( W_t \notin K ) = 2 Q_t ( \ds1_{K^c} )(x) .
\]
\end{lemma}
\begin{proof}
Let $(\mathcal F_t)$ be the 
natural filtration of the Brownian
motion. Fix $t>0$. By the strong Markov property
\begin{equation}\label{eq:reflectionprinciplestep}
\P_x ( W_t \notin K \mid \mathcal F_\sigma )
= u ( \sigma , W_\sigma ) ,
\end{equation}
where 
\[
u(s,y) = \ds1 \{ s < t\} \, \P_y ( W_{t-s} \notin K ) .  
\]
Let $y\in \partial K$, since $K$ is convex it admits a supporting 
hyperplane $H$ at $y$. Let $H_+$ be the halfspace
delimited by $H$ containing $K$. Then for any $u>0$
\[
\P_y ( W_u \notin K ) \geq \P_y ( W_u \notin H_+ ) = \frac 12 .
\]
Equality~\eqref{eq:reflectionprinciplestep} thus yields
\[
\P_x ( W_t \notin K \mid \mathcal F_\sigma )
\geq \frac 12 \, \ds1 \{ \sigma < t \} ,
\]
almost surely. Taking expectation yields the result. 
\end{proof}
We also need the following elementary estimate 
for the heat semigroup. 
\begin{lemma}\label{lem:heatsemiestimate}
For any $s\geq 0$
\[
\int_K Q_s ( \ds1_{K^c} ) \, dx \leq \sqrt{s} \, \mathcal H^{n-1} ( \partial K ) ,
\]
where $\mathcal H^{n-1} ( \partial K )$ is the Hausdorff measure
of the boundary of $K$.
\end{lemma}
\begin{proof}
Let $\phi(s) = \int_K Q_s ( \ds1_{K^c} ) \, dx $.
Then by definition of the heat semigroup and Stokes' formula
\[
\phi'(s) = \frac 12 \int_K \Delta Q_s ( \ds1_{K^c} ) \, dx
= \frac 12 \int_{\partial K} \langle \nabla Q_s ( \ds1_{K^c} ) (x) , \nu(x) \rangle \, 
\mathcal H^{n-1} (dx) ,
\]
for every $s>0$ and where 
$\nu (x)$ is an outer unit normal vector at point $x$.
On the other hand an elementary computation 
shows that for every $s>0$
\begin{equation}\label{eq:qtlip}
\vert \nabla Q_s ( \ds1_{K^c} ) \vert \leq s^{-1/2} , 
\end{equation}
pointwise. We thus obtain
\[
\vert \phi'(s) \vert \leq \frac{ \mathcal H^{n-1} ( \partial K ) }{ 2 \sqrt s} ,
\]
for every $s>0$. Integrating this inequality
between $0$ and $s$ yields the result. 
\end{proof}
\begin{proposition}\label{prop:w1totv}
Let $T,S$ be integer multiples of $\eta$.
Then
$$\mathrm{TV} ( X_{T+S} , \overline X_{T+S} ) \leq \frac{3 \E |X_T - \overline X_T|}{\sqrt{S}}
+ \mathrm{TV} (X_T,\mu) + 4 \, \sqrt{S} \, \mathcal H^{n-1} ( \partial K ) \, \vert K \vert^{-1} .$$
\end{proposition}
\begin{proof}
We use the coupling by reflection again. 
Fix $x$ and $x'$ in $K$. 
Let $(X_t)$ and $(X'_t)$ be two Brownian motions reflected 
at the boundary of $K$ starting from $x$ and $x'$ respectively, such that the 
underlying Brownian motions $(W_t)$ and $(W_t')$ are coupled by reflection, 
just as in the proof of Lemma~\ref{lem:mixing}. 
Let $(\overline{X'}_t)$ be the discretization of $(X'_t)$, 
namely the solution of the Skorokhod problem for the process
$\left( W'_{\eta \lfloor t/\eta \rfloor} \right)$.
Let $S$ be a integer multiple of $\eta$. Obviously,
if $(X_t)$ and $(X'_t)$ have merged before time $S$
and in the meantime neither $(X_t)$ nor $(X'_t)$
has hit the boundary of $K$ then 
\[
X_S = X'_S = \overline {X'}_S .
\] 
Therefore, letting $\tau$ be the first time
$X_t =X'_t$ and $\sigma$ and $\sigma'$ 
be the first times $(X_t)$ and 
$(X'_t)$ hit the boundary of $K$, respectively,
we have
\begin{equation}\label{eq:stepw1totv}
\P ( X_S \neq \overline X'_S ) 
\leq \P ( \tau > S ) + \P ( \sigma < S ) 
+ \P ( \sigma' < S ) ,
\end{equation}
As we have seen before, the coupling time $\tau$ satisfies
\[
\P ( \tau > S ) \leq \frac{ \vert x - x' \vert }{ \sqrt {2\pi S} } . 
\]
On the other hand Lemma~\ref{lem:reflectionprinciple} gives
\[
\P ( \sigma < S ) \leq 2 \, Q_S ( \ds1_{K^c} ) (x) ,
\]
and similarly for $\sigma'$. 
Notice also that the estimate~\eqref{eq:qtlip}
implies that
\[
Q_S ( \ds1_{K^c} ) (x')
\leq Q_S ( \ds1_{K^c} ) (x) +  \frac{ \vert x - x' \vert }{ \sqrt S }  .
\]
Plugging everything back into~\eqref{eq:stepw1totv} yields
\begin{equation}\label{eq:step2w1totv}
\P ( X_S \neq \overline X_S' ) 
\leq \frac{ 3 \vert x - x'\vert }{ \sqrt S } + 4 \, Q_S ( \ds1_{K^c} ) (x) .
\end{equation}
Now let $T$ and $S$ be two integer multiples of $\eta$
and assume that $(X_t)$ and $(\overline X_t)$ start from $0$ 
and are coupled using 
the same Brownian motion up to time $T$, 
and using the reflection 
coupling between time $T$ and $T+S$. 
Then, by Markov property 
and~\eqref{eq:step2w1totv} we get
\[
\P ( X_{T+S} \neq \overline X_{T+S} \mid \mathcal F_T ) \leq  
\frac{ 3 \vert X_T - \overline X_T \vert }{ \sqrt S }
+ 2 Q_S ( \ds1_{K^c} ) (X_T) . 
\]
Now we take expectation,
and observe that by Lemma~\ref{lem:heatsemiestimate}
\[
\begin{split}
\E \left[ Q_S ( 1_{K^c} ) (X_T) \right] 
& \leq \mathrm{TV} ( X_T , \mu ) + \int_K Q_S ( \ds1_{K^c} ) \, d \mu \\
& \leq \mathrm{TV} ( X_T , \mu ) + \sqrt S 
\,  \mathcal H^{n-1} (\partial K) \, \vert K \vert^{-1}.
\end{split} 
\]
Putting everything together we get the result. 
\end{proof}
\subsection{Proof of the main result} \label{sec:finishingtheproof1}
Let $S,T$ be integer multiples of $\eta$. Writing
\[
\mathrm{TV} ( \overline X_{T+S} , \mu ) 
\leq \mathrm{TV} ( \overline X_{T+S} , X_{T+S} )
+ \mathrm{TV} ( X_{T+S} , \mu )
\]
and using Proposition~\ref{prop:w1estimate} and Proposition~\ref{prop:w1totv} yields
\begin{equation}\label{eq:puttingeverythingtogether}
\begin{split}
\mathrm{TV} ( \overline X_{T+S} , \mu ) 
& \leq
C \left( \eta \log(T / \eta) \right)^{1/4} n^{3/4} \,  M^{1/2} \, T^{ 1/2 } \, S^{-1/2} 
+ 2\, \mathrm{TV} (X_T,\mu) \\
& + 4 \, S^{1/2} \, \mathcal H^{n-1} ( \partial K ) \, \vert K \vert^{-1} . 
\end{split}
\end{equation}
For sake of simplicity let us assume that $K$ contains the Euclidean 
ball of radius $1$, and let us aim at a result depending only on the diameter $R$
of $K$. So we shall use the trivial estimates
\[
m \leq R  , \quad  M \leq \frac 1r \leq 1 ,
\]
together with the less trivial but nevertheless true
\[
\mathcal H^{n-1} ( \partial K ) \leq n \vert K \vert . 
\]
Next we use Proposition~\ref{prop:mixing} 
to bound $\mathrm{TV} ( X_T , \mu )$ and~\eqref{eq:puttingeverythingtogether}
becomes
\[
\mathrm{TV} ( \overline X_{T+S} , \mu ) \leq
C \left( \left( \eta \log(T / \eta) \right)^{1/4} n^{3/4} \, T^{ 1/2 } \, S^{-1/2} 
+ e^{ - T /2 R^2 } \right) + 4 \, n \, S^{1/2} . 
\]
Given a small positive constant $\epsilon$,
we have to pick $S,T,\eta$ so that the 
right--hand side of the previous inequality
equals $\epsilon$. So we need to take
\[
S \approx \frac{ \epsilon^2 }{ n^2 } , \quad T \approx R^2 \, \log ( 1/\epsilon) ,
\]
and to choose $\eta$ so that 
\[
\frac \eta T \, \log \left( \frac T \eta \right)
\approx \frac { \epsilon^8 }{n^7 R^6 \log ( 1/ \epsilon )^3 }
\]
Since for small $\xi,\zeta$ we have 
\[
\xi \log (1/\xi) \approx \zeta \quad \Leftrightarrow 
\quad \xi \approx \frac{ \zeta }{ \log ( 1 / \zeta ) } ,
\]
and assuming that $R$ and $1/\epsilon$ are at most polynomial 
in $n$, we obtain 
\[
\eta \approx  
\frac { \epsilon^8 }{R^4 n^7 \log (n)^3 } .
\]
To sum up: 
Let $(\xi_k)$ be a sequence of i.i.d. standard Gaussian vectors,
choose the value of $\eta$ given above and 
run the algorithm
\[
\left\{
\begin{array}{l}
\overline X_0 = 0 \\
\overline X_{k+1} = \mathcal P_K \left( \overline X_k + \sqrt{ \eta } \,  \xi_{k+1} \right) 
\end{array}
\right. 
\]
for a number of steps equal to
\[
N = \frac{T+S}{\eta} \approx \frac { R^6 \, n^7 \, \log (n)^4 }{ \epsilon^8 } .
\]
Then the total variation between $\overline X_N$ and the uniform measure on $K$ 
is at most $\epsilon$.
\section{The general case} \label{sec:general}
In the previous section we viewed LMC (for a constant function $f$) as a discretization of reflected Brownian motion $(X_t)$ defined by $dX_t = dW_t - \nu_t L(dt)$ and $X_0=0$. In this section $(X_t)$ is a slightly more complicated process: it is a diffusion reflected at the boundary of $K$. More specifically $(X_t)$ 
\begin{equation} \label{eq:diffdef}
\begin{split}
& X_t \in K , \quad \forall t \geq 0 \\
& dX_t = dW_t - \frac{1}{2} \nabla f(X_t) dt - \nu_t L(dt) ,
\end{split}
\end{equation}
where $L$ is a measure supported on $\{ t \geq 0 : \ X_t \in \partial K \}$
and $\nu_t$ is an outer unit normal at $X_t$ for any such $t$. 
Recall the definition of LMC~\eqref{eq:LMC}, let us couple it 
with the continuous process $(X_t)$ as follows. Let $(Y_t)$ be a process
constant on each interval $[k\eta , (k+1)\eta)$ and satisfying
\begin{equation} \label{eq:LMC2}
Y_{(k+1)\eta} = \cP_K \left( Y_{k\eta} + W_{(k+1) \eta} - W_{k\eta} - 
\frac{\eta}{2} \nabla f( Y_{k\eta} ) \right) ,
\end{equation}
for every integer $k$. The purpose of this section 
is to give a bound on the total variation between $X_t$ and
its discretization $Y_t$. 
\subsection{Mixing time for the continuous process}
Since $\nabla f$ is assumed to be globally Lipschitz, the existence of the reflected diffusion is insured by~\cite[Theorem~4.1]{Tan79}. It\^o's formula then shows that $(X_t)$ is a Markov process whose generator is the operator $L$
\[
L h = \frac 12 \Delta h - \frac 12 \langle \nabla f , \nabla h \rangle 
\]
with Neumann boundary condition. Together with Stokes' formula, one can see that 
the measure 
\[
\mu ( dx ) = Z \, e^{- f(x)} \, 1_K (x) \, dx  
\]
(where $Z$ is the normalization constant) is the unique stationary
measure of the process, and that it is even reversible.
 
We first show that if $f$ is convex the mixing
time estimate of the previous section remains valid. Again
given a probability measure $\nu$ supported on $K$ we let 
$\nu P_t$ be the law of $X_t$ when $X_0$ has law $\nu$. 
\begin{lemma}\label{lem:mixing2}
If $f$ is convex then for every $x,x' \in K$ 
\[
\mathrm{TV} ( \delta_x P_t , \delta_{x'} P_t )
\leq \frac { \vert x'-x \vert }{ \sqrt{2\pi t} } . 
\]
\end{lemma}
\begin{proof}
As in the proof of Lemma~\ref{lem:mixing}, let $(X_t)$ and $(X_t')$ be two reflected diffusions starting from $x$ and $x'$ and such that the underlying Brownian motions are coupled by reflection. In addition to~\eqref{eq:something}, one also has 
$$\langle \nabla g ( X_t - X'_t ) , \nabla f(X_t) - \nabla f(X_t') \rangle \geq 0 ,$$
by convexity of $f$. The argument then goes through verbatim. 
\end{proof}
As in section~\ref{sec:mixing}, this lemma 
allows us to give the following bound on the mixing time
of $(X_t)$.
\begin{proposition}\label{prop:mixing2}
For any $t>0$
\[
\mathrm{TV} ( \delta_0 P_t , \mu ) 
\leq  C \, \min \left( m \,  t^{-1/2}  , e^{ -t / 2R^2 } \right) , 
\]
where $C$ is a universal constant. 
\end{proposition}
\subsection{A change of measure argument} \label{sec:change}
Again let $(X_t)$ be the reflected diffusion~\eqref{eq:diffdef}. Assume that
$(X_t)$ starts from $0$ and let $(Z_t)$ be the process 
\begin{equation} \label{eq:z_t}
Z_t = W_t - \frac{1}{2} \int_0^t \nabla f(X_s) \, ds .
\end{equation}
Observe that $(X_t)$ solves the Skorokhod problem 
for $(Z_t)$. Following the same steps as in the previous section we let 
\[
\overline Z_t = Z_{ \lfloor t / \eta \rfloor \eta }
\]
and we let $(\overline X_t)$ be the solution of the Skorokhod
problem for $(\overline Z_t)$. In other words $(\overline X_t)$
is constant on intervals of the form $[ k \eta , (k+1) \eta )$ 
and for every integer $k$
\begin{equation}\label{eq:oxt}
\overline X_{(k+1) \eta} = \cP_K 
\left( \overline X_{k \eta} + Z_{(k+1)\eta} - Z_{k\eta}  \right) ,
\end{equation}
Clearly $(\overline X_t)$ and $(Y_t)$ are different processes
(well, unless the potential $f$ is constant). 
However, we show in this subsection that 
using a change of measure trick similar to the
one used in~\cite{Dal14}, it is possible to bound 
the total variation distance between $\overline X_t$ and $Y_t$. 
Recall first the hypothesis
made on the potential $f$ 
\[
\vert \nabla f (x) \vert \leq L,
\quad \vert \nabla f(x) - \nabla f(y) \vert \leq \beta \vert x-y\vert , 
\quad \forall x,y\in K .  
\] 
\begin{lemma}\label{lem:trickdal}
Let $T$ be an integer multiple of $\eta$. Then 
\[
\mathrm{TV} ( \overline X_T , Y_T ) \leq 
\frac{ \sqrt{L \beta} } 2 \, 
\left( \E \left[ \int_0^T \vert X_s
- \overline X_s \vert \, ds \right] \right)^{1/2} .  
\]
\end{lemma}
\begin{proof}
Write $T=k\eta$. Given a
continuous path $(w_t)_{t\leq k\eta}$ 
we define a map
$Q$ from the space of sample paths to $\R$ by setting $Q(w) = x_k$ where 
$(x_i)$ is defined inductively as
\[
\begin{split}
x_0 &= 0 \\
x_{i+1} & = \cP_K 
\left( x_i + w_{(i+1)\eta} - w_{i\eta} 
- \frac{\eta}{2} \nabla f(x_i) \right) , \quad i\leq k-1. 
\end{split}
\]
Observe that with this notation we have $Y_{k\eta} = Q ( (W_t)_{t\leq k \eta} )$.
On the other hand, letting $(u_t)$ be the process
\[
u_t = \frac 12 \left( \nabla f ( \overline X_t ) - \nabla f ( X_t ) \right) ,
\]
letting $\tilde W_t = W_t + \int_0^t u_s \, ds$ and using equation~\eqref{eq:oxt},
it is easily seen that
\[
\overline X_{k\eta} = Q \left( (\tilde W_t)_{t\leq k \eta} \right) . 
\]
This yields the following inequality 
for the relative entropy of $\overline X_{k\eta}$ with respect
to $Y_{k\eta}$:
\begin{equation}\label{eq:entropy}
\mathrm H ( \overline X_{k\eta} \mid Y_{k\eta} ) \leq 
\mathrm H \left( (\tilde W_t)_{t\leq k\eta} \mid (W_t)_{t\leq k\eta} \right) .  
\end{equation}
Since $\tilde W$ is a Brownian motion 
plus a drift (observe that the process 
$(u_t)$ is adapted to the natural filtration of $(W_t)$)
it follows form Girsanov's formula, 
see for instance Proposition~1 in \cite{Leh13},
that
\[
\begin{split}
\mathrm H \left( (\tilde W_t)_{t\leq k\eta} \mid (W_t)_{t\leq k\eta} \right) 
& \leq \frac 12 \E \left[ \int_0^{k\eta} \vert u_t \vert^2 \, dt \right] \\
& =\frac 18 \E \left[  \int_0^{k\eta} \vert \nabla f ( \overline X_t ) 
- \nabla f ( X_t ) \vert^2 \, dt \right] . 
\end{split} 
\]
Plugging this back in~\eqref{eq:entropy} and
using the hypothesis made on $f$ we get
\[
\mathrm H ( \overline X_{k\eta} \mid Y_{k\eta} ) \leq 
\frac {L \beta}4 
\E \left[ \int_0^{k\eta} \vert X_t - \overline X_t \vert \, dt \right] .
\]
We conclude by Pinsker's inequality. 
\end{proof} 
The purpose of the next two subsections is 
to estimate the transportation and total 
variation distances between $X_t$ and 
$\overline X_t$.
\subsection{Estimation of the Wasserstein distance}
First we extend Lemma~\ref{lem:local} and Lemma~\ref{lem:doob} to the general case.
\begin{lemma}\label{lem:localdrift}
We have, for all $t>0$
\[
\E \left[ \int_0^t h_K ( \nu_s ) \, L(ds) \right]  
\leq \frac{ (n+RL) t}{2} . 
\]
\end{lemma}
\begin{proof}
As in the proof of Lemma~\ref{lem:local}, 
It\^{o}'s formula yields
\[
2 \int_0^t h_K ( \nu_s ) \, L(ds)
= 2 \int_0^t \langle X_s , d W_s \rangle 
- \int_0^t \langle X_s , \nabla f (X_s) \rangle \, ds  
+ n t + \vert X_0 \vert^2 - \vert X_t \vert^2 . 
\]
Assume that $X_0 =0$, note that 
the first term is a martingale 
and observe that $\vert \langle X_s , \nabla f(X_s) \rangle\vert \leq  R L$
by hypothesis. Taking expectation in the previous 
display, we get the result.
\end{proof}
Recall the definition of the process $(Z_t)$:
\[
Z_t = W_t - \frac 12 \int_0^t \nabla f ( X_s ) \, ds ,
\]
and recall that $(\overline Z_t)$ is its discretization: 
$\overline Z_t = Z_{ \eta \lfloor t/\eta \rfloor}$. 
\begin{lemma}\label{lem:doobdrift}
There exists a universal constant $C$ such that
\[
\E \left[ \sup_{[0,t]} \Vert Z_s - \overline Z_s \Vert_K \right]
\leq C M n^{1/2} \eta^{1/2} \log( t/\eta)^{1/2} + \frac {\eta L}{2r} . 
\]
\end{lemma}
\begin{proof} 
Since for every $x\in \R^n$
\[
\Vert \nabla f(x) \Vert_K \leq \frac 1r \, \vert \nabla f (x) \vert \leq \frac Lr,
\]
we have
\[
\begin{split}
\Vert Z_t - \overline Z_{t} \Vert_K 
& \leq \Vert W_t - \overline W_{t} \Vert_K 
+\frac 12 \int_{\lfloor t/\eta \rfloor \eta}^t \Vert \nabla f ( X_t ) \Vert_K \, dt \\
& \leq \Vert W_t - \overline W_t \Vert_K + \frac{ \eta L}{2r} ,
\end{split}
\]
for every $t>0$. 
Together with Lemma~\ref{lem:doob}, we get the result. 
\end{proof}
As in section~\ref{sec:w1}, combining these two
lemmas together yields the following estimate. 
\begin{proposition}\label{prop:w1estimate2}
For every time $T$, we have
\[
\E \left[  \vert X_T - \overline X_T \vert \right] 
\leq C\, \left(  C_1  \, \left( \eta  \log( T / \eta) \right)^{1/4} T^{1/2} 
+ C_2 \, \eta^{1/2} T^{1/2} \right) , 
\]
where $C$ is a universal constant and where
\[
\begin{split}
C_1 = C_1 (K,f) =  n^{3/4} M^{1/2} + n^{1/2} R^{1/2} M^{1/2} L^{1/2}  \\
C_2 = C_2 (K,f) = n^{1/2} r^{-1/2}  L^{1/2} + R^{1/2} r^{-1/2} L .
\end{split}
\]
\end{proposition}
\subsection{From Wasserstein distance to total variation}\label{sec:t1totv2}
Unless $f$ is constant, the diffusion $(Z_t)$ does not satisfy Lemma~\ref{lem:reflectionprinciple}
so we need to proceed somewhat differently from what was done in section~\ref{sec:w1totv}. 
We start with a simple lemma showing that $\mu$ does not put too much mass close 
to the boundary of $K$.
\begin{lemma} \label{lem:distancetoboundary}
Let $\gamma > 0$. One has
$$
\mu(\{x \in K, d(x,\partial K) \leq \gamma \} ) \leq \frac{ (n+ R L) \gamma} r  . 
$$
\end{lemma}
\begin{proof}
Define
$$
K_{\gamma} := \{x \in K; ~ d(x, \partial K) \geq \gamma \}.
$$
Let $\mathbb{B}^n$ be the Euclidean ball, since $K$ contains $r \mathbb B^n$
and is convex we have
$$\left(1 - \frac{\gamma}{r} \right) K + \frac{\gamma}{r} r \mathbb{B}^n \subset K ,$$
hence
$$
\left(1 - \frac{\gamma}{r} \right) K \subset K_{\gamma} .
$$
Clearly this implies:
$$
\int_{K_{\gamma}} e^{-f(x)} \, dx \geq \left(1 - \frac \gamma r \right)^n 
\int_K e^{ - f \left( ( 1 - \gamma/r)y \right) }\, dy.
$$
Since $f$ is Lipschitz with constant $L$ one also has
$$
f \left( (1 - \gamma / r) y \right) \leq f(y) - \frac{ L \gamma |y| } r 
\leq f(y) - \frac{ R L \gamma } r
$$
for every $y\in K$. 
Combining the last two displays, we obtain
\begin{align*}
\int_{K_{\gamma}} \exp(-f(x)) \, dx 
~& \geq \left(1 - \frac \gamma r \right)^n e^{- R L \gamma / r} \int_{K} e^{-f(x)} \, dx \\
~& \geq \left(1 - \frac{n \gamma } r - \frac{ R L \gamma } r \right) \int_{K} e^{-f(x)} \, dx , 
\end{align*}
which is the result. 
\end{proof}
Here is a simple bound on the speed of a Brownian motion with drift.
\begin{lemma} \label{lem:escape}
Let $(W_t)$ be a standard Brownian motion (starting from $0$), 
let $(v_t)$ an adapted drift satisfying 
$|v_t| \leq L$ (almost surely), and $(Z_t)$ the process given by
$$Z_t = W_t + \int_0^t v_s ds .$$
Then for every $t>0$ and every $\gamma >0$
$$
\P \left ( \sup_{s\in [0,t]} |Z_s| > \gamma \right ) \leq 
\frac{ \sqrt{ n t } + L t }{ \gamma } .
$$
\end{lemma}

\begin{proof}
By the triangle inequality and since $|v_t| < L$, we have
\[
\vert Z_s \vert \leq \vert W_s \vert + L s ,
\] 
for any $s$. 
Now the process $( \vert W_s \vert + Ls )$ is non--negative submartingale
so by Doob's maximal inequality
\[
\P \left ( \sup_{s\in [0,t]} |Z_s| > \gamma \right ) \leq 
\frac{ \E \left[ \vert W_t \vert + L t \right] }{ \gamma } . 
\]
Since $\E [ \vert W_t \vert ] \leq \sqrt{ nt }$, we get the result. 
\end{proof}

\begin{proposition}\label{prop:w1totv2}
Let $T$ and $S$ be integer multiples of $\eta$. 
We have
\[
\mathrm{TV} ( X_{T+S} , \overline X_{T+S} ) 
\leq C \, \left( W(T) S^{-1/2} + \mathrm{TV} (X_T,\mu) +  C_3 \, S^{1/4}
+ C_4 \, S^{1/2}
+ C_5 \, W(T)^{1/2} \right)  , 
\]
where $C$ is a universal constant, 
$W(T)$ is the bound obtained in Proposition~\ref{prop:w1estimate2}
and
\[
\begin{split}
C_3 & = n^{1/4} R^{1/2} r^{-1/2} L^{1/2} + n^{3/4} r^{-1/2} \\
C_4 & = R^{1/2} r^{-1/2} L + n^{1/2} r^{-1/2} L^{1/2} \\
C_5 & = R^{1/2} r^{-1/2} L^{1/2} + n^{1/2} r^{-1/2} .
\end{split}
\]
\end{proposition}
\begin{proof}
The proof follows similar lines to those of the proof of Proposition \ref{prop:w1totv}, but the drift term requires some additional bounds which will be provided by the previous two lemmas. 

We begin with fixing two points $x, x' \in K$ and we consider the two associated diffusions processes $(X_t)$ and $(X_t')$, which start from the points $x$ and $x'$ respectively, such that the underlying Brownian motions are coupled by reflection. In other words, those processes satisfy equations \eqref{eq:coupledprocs1} and \eqref{eq:coupledprocs2} with the additional drift term.

In analogy with the process $(Z_t)$, let $(Z'_t)$ be the process
\[
Z'_t =  W'_s - \frac 12 \int_0^t \nabla f ( X_s') \, ds ,
\] 
let $\overline{Z}'_t = Z'_{\eta \lfloor t/\eta \rfloor}$ and let 
$(\overline{X}'_t)$ be the solution 
of the Skorokhod problem for $(\overline{Z}'_t)$. 
We proceed as in the proof of Proposition \ref{prop:w1totv}, letting $\tau$ be the coupling time of $(X_t)$
and $(X_t')$ and letting $\sigma$ and $\sigma'$ 
be the first time $(X_t)$ and $(X_t')$ hit the boundary of $K$, we have that
\[
\P ( X_S \neq \overline {X}'_S ) 
\leq \P ( \tau > S ) + \P ( \sigma \leq S ) 
+ \P ( \sigma' \leq S ).
\]
Moreover the coupling time $\tau$ still satisfies
\[
\P ( \tau > S ) \leq \frac{ \vert x - x' \vert }{ \sqrt {2\pi S} }.
\]
Now fix $\gamma >0$ and observe that if $d(x,\partial K) > \gamma$, 
then $\sigma$ is at least the first time the process
\[
W_t - \frac 12 \int_0^t \nabla f ( X_s ) \, ds  
\] 
hits the sphere centered at $x$ of radius $\gamma$. 
So, by Lemma~\ref{lem:escape},
\[
\P ( \sigma \leq S ) \leq \frac{ \sqrt{nS} + L S }{ \gamma } 
+ \ds1_{ \{ d (x,\partial K ) \leq\gamma \} } .
\] 
There is a similar inequality for $\sigma'$
and we obtain 
\[
\begin{split}
\P ( X_S \neq \overline{X}'_S )
& \leq \frac{ \vert x - x' \vert }{ \sqrt {2\pi S} } + \frac { 2 \sqrt{nS} + 2 LS } \gamma
+ \ds1_{ \{ d (x,\partial K ) \leq \gamma \} } + \ds1_{ \{ d (x',\partial K ) \leq\gamma \} } \\
& \leq \frac{ \vert x - x' \vert }{ \sqrt {2\pi S} } + \frac { 2 \sqrt{nS} + 2 LS } \gamma
+ 2\, \ds1_{ \{ d (x,\partial K ) \leq 2 \gamma \} } 
+ \ds1_{ \{ \vert x-x' \vert ) \geq \gamma \} } .
\end{split}
\]
So if $T$ and $S$ are two integer multiples of $\eta$,
if $(X_t)$ and $(\overline {X}_t)$ start from $0$, 
are coupled using the same Brownian motion up to time $T$, 
and using the reflection coupling between time $T$ and $T+S$,
then we have  
\[
\begin{split}
\P ( X_{T+S} \neq \overline X_{T+S} ) \leq  
\frac{ \E \left[ \vert X_T - \overline X_T \vert \right] }{ \sqrt {2\pi S} }
& + \frac { 2 \sqrt{nS} + 2 LS } \gamma  
+ 2\, \P \left(  d ( X_T , \partial K ) \leq 2 \gamma \right) \\
& + \P \left( \vert X_T-\overline X_T \vert \geq \gamma \right)  .
\end{split}
\]
By Lemma~\ref{lem:distancetoboundary},
\[
\begin{split}
\P \left( d( X_T , \partial K ) \leq 2 \gamma \right) 
& \leq \mu \left( d(x,\partial K) \leq 2 \gamma \right) + \mathrm{TV} ( X_T , \mu ) \\
& \leq \frac{2(R L + n)\gamma}{r}  + \mathrm{TV} ( X_T , \mu ),
\end{split}
\]
and an application of Markov's inequality gives
\[
\P ( \vert X_T - \overline X_T \vert \geq \gamma ) 
\leq \frac{ \E[ \vert X_T - \overline X_T \vert ] } \gamma .
\]
Combining the last three displays together, we finally obtain
\[
\begin{split}
\P ( X_{T+S} \neq \overline X_{T+S} ) \leq  
\frac{ \E \left[ \vert X_T - \overline X_T \vert \right] }{ \sqrt {2\pi S} }
& + \frac { 2 \sqrt{nS} + 2 LS } \gamma  
+ \frac{4(R L + n)\gamma}{r}  \\
& + 2 \, \mathrm{TV} ( X_T , \mu ) + \frac{ \E[ \vert X_T - \overline X_T \vert ] } \gamma .
\end{split}
\]
Optimizing over $\gamma$ and using Proposition~\ref{prop:w1estimate2} 
yields the desired inequality.
\end{proof}
\subsection{Proof of Theorem \ref{th:mainresult}} \label{sec:finishingtheproof}
This subsection contains straightforward calculations to help the reader put 
together the results proven above. Hereafter, to simplify notation, the constants $c, C$ will represent positive universal constants whose value may change between different appearances.

Let $T$ and $S$ be integer multiples of $\eta$
and write
$$ 
\mathrm{TV}( Y_{T+S} , \mu ) 
\leq \mathrm{TV} ( Y_{T+S} , \overline X_{T+S} ) 
+ \mathrm{TV} ( \overline X_{T+S} , X_{T+S} ) + \mathrm{TV}( X_{T+S} , \mu) . 
$$
Again, we will not try to give an optimal result
in terms of all the parameters. So assume for simplicity 
that $K$ contains the Euclidean ball 
of radius $1$ so that $r$ is replaced by $1$ in constants $C_2,C_3,C_4$ and $C_5$.
Also let 
\[
n_\star = \max ( n , RL , R \beta ) .
\] 
Keeping in mind that $S$ shall be chosen to be rather small (hence assuming $S \leq 1$), 
Proposition~\ref{prop:w1totv2} is easily seen to imply 
that 
\[
\frac 1 C \, \mathrm{TV}( X_{T+S} , \overline X_{T+S} ) 
\leq W(T) S^{-1/2} + \mathrm{TV} ( X_T , \mu ) + n_\star \, S^{1/4} + 
( n_\star \,  W(T) )^{1/2} ,
\]
Together with Lemma~\ref{lem:trickdal} and Proposition~\ref{prop:mixing2} we get
\[
\frac 1 C \, \mathrm{TV}( Y_{T+S} , \mu ) 
\leq  ( L \beta T + n_\star )^{1/2} W(T)^{1/2} 
+ W(T) S^{-1/2} + n_\star \,  S^{1/4} + e^{-T/2R^2} . 
\]
Fix $\epsilon >0$ and choose 
\[
S = n_\star^{-4}\, \epsilon^4 , \quad  
T = R^2 \log(1/\epsilon) .
\]
Then it is easy to see that it is enough to
pick $\eta$ small enough so that 
\[
W(T) < C n_\star^{-2} \, \epsilon^3 \log ( 1/\epsilon )^{-1} ,
\]
to ensure $\mathrm{TV} ( X_{T+S} , \mu ) \leq C \epsilon$. 
Now Proposition~\ref{prop:w1estimate2} clearly yields
\[
W(T) < C n_\ast \left( \eta \log( T/ \eta) \right)^{1/4} T^{1/2} . 
\]
Recall that $T = R^2 \log( 1/\epsilon)$ and observe that 
\[
\eta \leq c \frac{ \epsilon^{12} }{ n_\star^{12} \, R^4 \, \max(\log(n), \log(R), \log(1/\eps))^7 }
\]
suits our purpose. Lastly for this choice of $\eta$ the number of 
steps in the algorithm is 
\[
N = \frac {T+S} \eta \leq C \frac{ n_\star^{12} \, R^6 \, \max(\log(n), \log(R), \log(1/\eps))^8 }{ \epsilon^{12} } .
\]
\section{Experiments} \label{sec:exp}
Comparing different Markov Chain Monte Carlo algorithms is a challenging problem in and of itself. Here we choose the following simple comparison procedure based on the volume algorithm developed in \cite{CV14}. This algorithm, whose objective is to compute the volume of a given convex set $K$, procedes in phases. In each phase $\ell$ it estimates the mean of a certain function under a multivariate Gaussian restricted to $K$ with (unrestricted) covariance $\sigma_{\ell} \mathrm{I}_n$. Cousins and Vempala provide a Matlab implementation of the entire algorithm, where in each phase the target mean is estimated by sampling from the truncated Gaussian using the hit-and-run (H\&R) chain. We implemented the same procedure with LMC instead of H\&R, and we choose the step-size $\eta = 1/ (\beta n^2)$, where $\beta$ is the smoothness parameter of the underlying log-concave distribution (in particular here $\beta = 1 / \sigma_{\ell}^2$). The intuition for the choice of the step-size is as follows: the scaling in inverse smoothness comes from the optimization literature, while the scaling in inverse dimension squared comes from the analysis in the unconstrained case in \cite{Dal14}.

\vspace{-1.7in}
\begin{tabular}{cc}
\hspace{-1.7in} \includegraphics[width=0.85\linewidth]{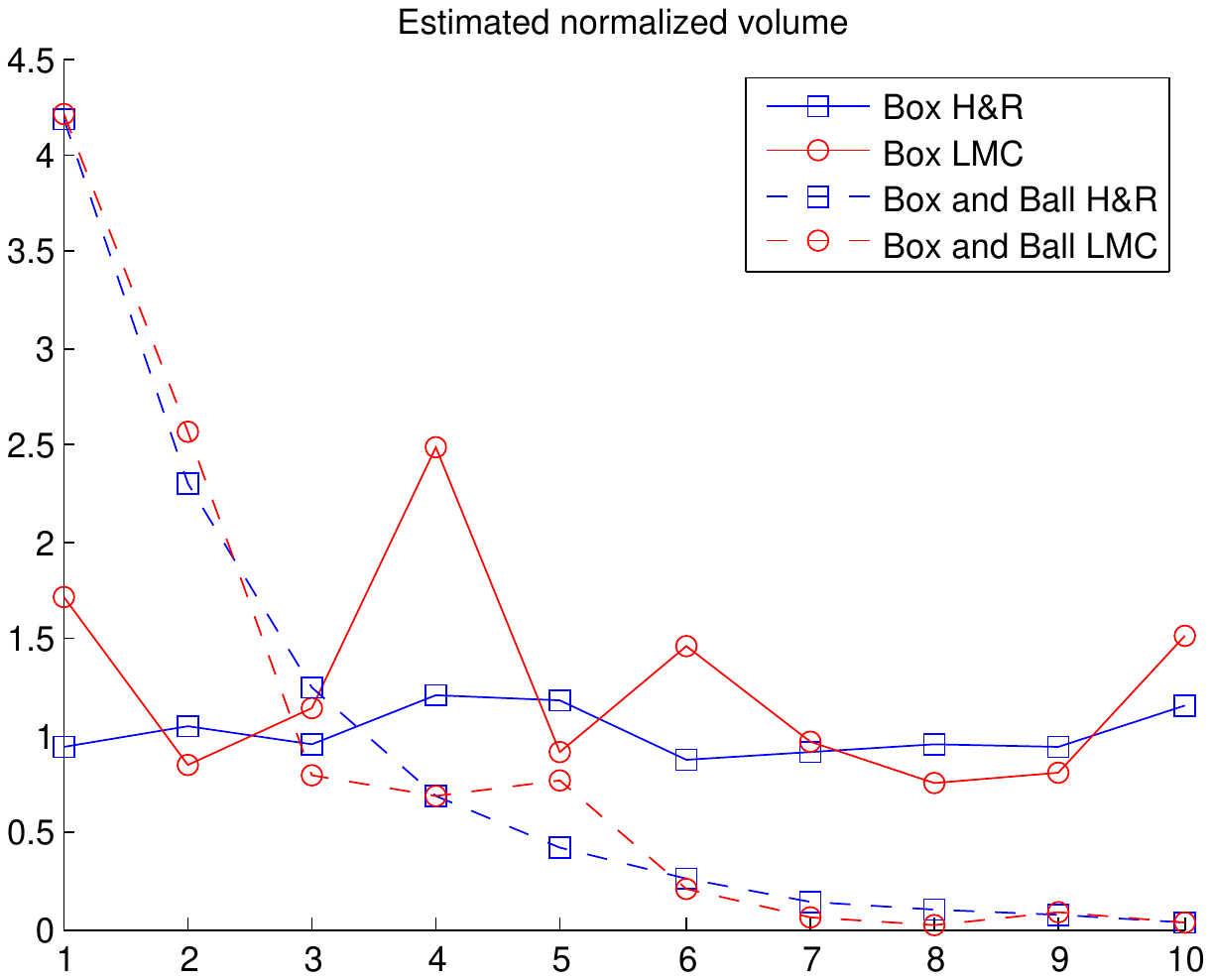} & \hspace{-2in}
\includegraphics[width=0.85\linewidth]{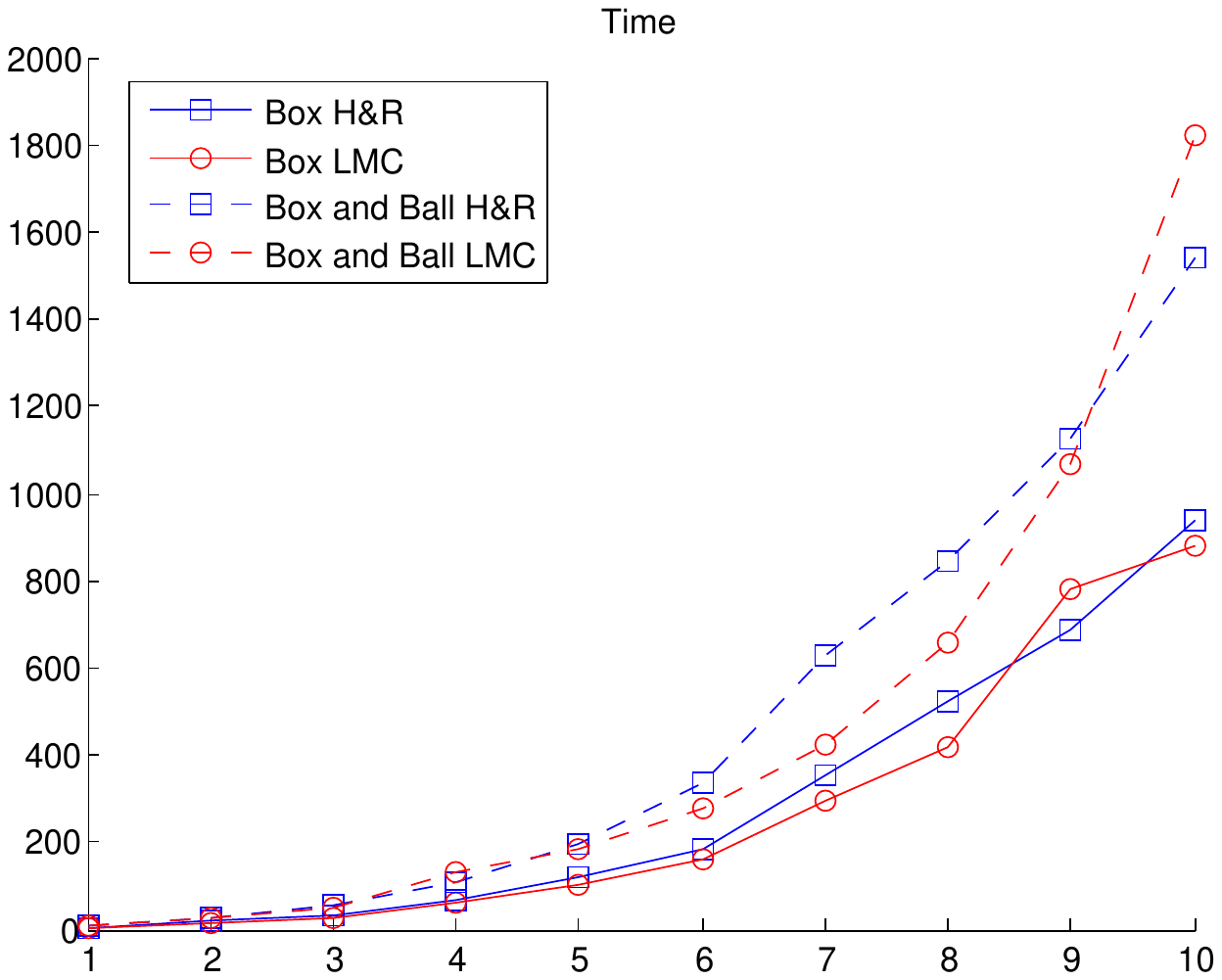} 
\end{tabular}
\vspace{-1.8in}

We ran the volume algorithm with both H\&R and LMC on the following set of convex bodies: $K=[-1,1]^n$ (referred to as the ``Box'') and $K=[-1,1]^n \cap \left(\frac{\sqrt{n}}{2} \mathbb{B}^n\right)$ (referred to as the ``Box and Ball''), where $n=10 \times k, k=1,\hdots,10$. The computed volume (normalized by $2^n$ for the ``Box'' and by $0.2 \times 2^n$ for the ``Box and Ball'') as well as the clock time (in seconds) to terminate are reported in the figure above.
From these experiments it seems that LMC and H\&R roughly compute similar values for the volume (with H\&R being slightly more accurate), and LMC is almost always a bit faster. These results are encouraging, but much more extensive experiments are needed to decide if LMC is indeed a competitor to H\&R in practice.

\bibliographystyle{plainnat}
\bibliography{bibsampling}
\end{document}